\documentclass[12pt,reqno]{amsart}

\usepackage{amssymb}
\usepackage{etex} 
\usepackage{tikz-cd}
\usepackage{mathtools}
\usepackage{xifthen}
\usepackage{stmaryrd}  
\usepackage{enumitem}

\usepackage[left=2.5cm,right=2.5cm,top=2.7cm,bottom=2.8cm]{geometry}
\usepackage{microtype}

{\end{description}}

\definecolor{linkblue}{RGB}{1,1,190}
\definecolor{citegreen}{RGB}{1,170,40}
\usepackage[linkcolor=linkblue
           ,citecolor=citegreen
           ,ocgcolorlinks        
           ,bookmarksopen=True,  
           ]{hyperref}
\makeatletter
  \AtBeginDocument{
    \hypersetup{
      pdftitle  = {\@title},
      pdfauthor = {\authors}     
    }
  }
\makeatother

\usepackage{latexsym}
\usepackage{amscd}
\usepackage{epsfig}
\usepackage{latexsym, epsfig, epic}
\usepackage{amsmath}
\usepackage{amssymb}
\usepackage{amsthm}
\usepackage{pgf}
\usepackage{tikz}
\usepackage{tikz-cd}
\usetikzlibrary{matrix,arrows,backgrounds,shapes.misc,shapes.geometric,patterns,calc,positioning}
\usetikzlibrary{calc,shapes}
\usepackage{wrapfig}
\usepackage{epsfig}
\usepackage{amsmath}
\usepackage{amsfonts}
\usepackage{amssymb}
\usepackage{graphicx}
\usepackage{multicol}
\usepackage[all]{xy}
\usepackage{xcolor}
\usepackage{color}
\usepackage{float}

\usepackage{graphics}

\def\semidirprod{\begin{picture}(8,8)\qbezier(2,0.5)(5,3.5)(8,6.5)\qbezier(2,6.5)(5,3.5)(8,0.5)\put(2,0.5){\line(0,1){6}}\end{picture}}

\newcommand{\tilT}{{\widetilde{T}}}
\newcommand{\tilX}{{\widetilde{X}}}
\newcommand{\tilY}{{\widetilde{Y}}}

\DeclareMathOperator{\End}{End}

\DeclareMathOperator{\Ext}{Ext}
\DeclareMathOperator{\Hom}{Hom}
\DeclareMathOperator{\mmod}{mod}
\DeclareMathOperator{\add}{add}
\DeclareMathOperator{\ind}{ind}

\newcommand{\ra}{\rightarrow}
\newcommand{\la}{\leftarrow}

\newcommand{\xra}{\xrightarrow}

\newcommand{\Aa}{\mathbb{A}}
\newcommand{\tAa}{\tilde{\mathbb{A}}}
\newcommand{\Aaa}{\mathcal{A}}
\newcommand{\Bb}{\mathcal{B}}
\newcommand{\Cc}{\mathcal{C}}
\newcommand{\Bp}{\mathcal{B}_p}
\newcommand{\Bq}{\mathcal{B}_q}
\newcommand{\Db}{\mathcal{D}^b}

\newcommand{\Tt}{\mathcal{T}}

\newcommand{\Rr}{\mathcal{R}}
\newcommand{\Ss}{\mathcal{S}}

\newcommand{\Uu}{\mathcal{U}}

\newtheorem{teo}{Theorem}[section]
\newtheorem{lema}[teo]{Lemma}
\newtheorem{teorema}[teo]{Theorem}

\newtheorem{proposicion}[teo]{Proposition}

\newtheorem{ejemplo}[teo]{Example}
\newtheorem{remark}[teo]{Remark}

\newtheorem{notation}[teo]{Notation}

\newtheorem{definicion}[teo]{Definition}

\newtheorem{teorema*}{Theorem}
\newtheorem{corolario*}{Corollary}

\usepackage{color}

\AtEndDocument{\bigskip{\footnotesize%
  \textsc{Facultad de Ingenier\'ia, Universidad Nacional de la Patagonia San Juan
Bosco, 9120 Puerto Madryn, Argentina} \par
 \textit{E-mail address}  \texttt{elsafer9@gmail.com} \par

  \textsc{University of Graz, Institute for Mathematics and Scientific Computing, NAWI Graz,, Heinrichstra{\ss}e 36, A-8010 Graz, Austria} \par
  \textit{E-mail address} \texttt{ana.garcia-elsener@uni-graz.at} \par

  \textsc{Centro Marplatense de Investigaciones Matem\'aticas. FCEyN, Universidad Nacional de Mar del Plata, CONICET. Dean Funes 3350, Mar del Plata, Argentina.} \par
  \textit{E-mail address} \texttt{strepode@mdp.edu.ar}

}}

\begin{document}

\title{$m$-cluster tilted algebras of Euclidean type}
\author{Elsa Fern\'andez, Ana Garcia Elsener \and Sonia Trepode}



\begin{abstract}

  We consider $m$-cluster tilted algebras arising from quivers of Euclidean type and we give necessary and sufficient conditions for those algebras to be representation finite. For the case $\tAa$, using the geometric realization, we get a description of representation finite type in terms of $(m+2)$-angulations. We establish which $m$-cluster tilted algebras arise at the same time from quivers of type $\tAa$ and $\Aa$. Finally, we characterize representation infinite $m$-cluster tilted algebras arising from a quiver of type $\tAa$, as $m$-relations extensions of some iterated tilted algebra of type $\tAa$.
\end{abstract}

\maketitle
\begin{center}
\textit{Dedicated to Jos\'e Antonio de la Pe\~na on the occasion of his 60th birthday}
\end{center}

\section{Introduction}\label{section introduction}
In this work we deal with $m$-cluster tilted algebras that were introduced by H. Thomas in \cite{T}. This class of algebras has been studied by several authors in the last years, see for example \cite{Ba}, \cite{Mu}, \cite{Tol}, \cite{Gub}, \cite{FPT}. We start our work with two very surprising facts about $m$-cluster tilted algebras for  $m$ greater or equal than two, that is, the type $Q$ is not well defined and the representation type is not preserved. This two facts hold true for cluster tilted algebras. Both results were proven by A. Buan, R. Marsh and I. Reiten in \cite{BMR} \cite{BMR2}.

In \cite{Gub}, V. Gubitosi applies the geometric model introduced by H. A. Torkildsen \cite{Tol} to describe the quivers and relations of $m$-cluster tilted algebras arising from a quiver of type $\tAa$. She shows examples of $m$-cluster tilted algebras arising at the same time from quivers of type $\tAa$ y $\Aa$. These examples show that the type $Q$ is not well defined for $m$-cluster tilted algebras for $m$ greater or equal than two. On the other hand, these $m$-cluster tilted algebras arising for quivers of type $\tAa$ and $\Aa$ are of finite representation type, showing that the representation type is not preserved.

Under this motivation we obtain our main result. In a first step we give a complete homological example illustrating the situation in \cite{Gub}. On the other hand, we prove that every iterated tilted algebra is an $m$-cluster tilted algebra for some $m$. Looking at the example \ref{Ejemplo motivacion}, we observe a similar behavior that the one occurring in the case of representation finite iterated tilted algebras of Euclidean type, see \cite{AHT}. In this paper, the authors show that representation finite iterated tilted algebra of Euclidean type  are characterized as the ones obtained as an endomorphism ring of a tilting complex with non zero direct summands in at least two different transjective components of the Auslander-Reiten quiver of its derived category. We obtain here a similar result. We characterize representation finite $m$-cluster tilted algebras in terms of the position of the direct summands of the $m$-cluster tilting object.

We denote by $\Cc^m(H)$ the $m$-cluster category defined by a hereditary algebra $H$, and by $\Gamma (\Cc^m(H))$ the Auslander-Reiten quiver of this category. See these definitions and the definition of $m$-cluster tilting object in section \ref{section-prelim}.

We are now in a position to state our main theorem.

\begin{teorema*} (Theorem \ref{theorem-main})
Let $\Cc^m(H)$ be the $m$-cluster category of Euclidean type and let $A= \End_{\Cc^m(H)}(\tilT)$ be an $m$-cluster tilted algebra. Then, $A$ is of finite representation type if and only if the $m$-cluster tilting object $\tilT$ has summands in at least two different transjective components of $\Gamma(\Cc^m (H))$.
\end{teorema*}

In section \ref{section A tilde}, we prove that the main theorem can be obtained from the geometric-combinatoric point of view in the case of $m$-cluster tilted algebras arising from a quiver of type $\tAa$. We recall the geometric realizations for the $m$-cluster categories and $m$-cluster tilted algebras arising from quivers of types $\Aa$ and $\tAa$ given in \cite{Ba,Mu,Tol,Gub}. We characterize which $m$-cluster tilted algebras arise at the same time from quivers of Dynkin type $\Aa$ and Euclidean type $\tAa$. This depends only on the existence of a root cycle, see Definition \ref{root cycle} and \cite[Def. 7.2]{Gub}. More precisely, we have the following result.

\begin{teorema*} (Theorem \ref{teorema A})
Let $A= k Q / I$ be an $m$-cluster tilted algebra arising from a quiver of type $\tAa$. Then, $A$ is an  $m$-cluster tilted algebra arising simultaneously from a quiver of type $\Aa$ if and only if $Q$ does not have a root cycle.
\end{teorema*}

We observe that this class of algebras are not all the representation finite $m$-cluster tilted algebras arising from a quiver of type $\tAa$.  We have from \cite[Theorem 7.16]{Gub} that there exist derived discrete algebras which appear as representation finite $m$-cluster tilted algebras arising from a quiver of type $\tAa$. Indeed, consider the algebra given by a root cycle where the number of clockwise oriented relations is not equal to the number of counterclockwise oriented relations and these numbers are congruent module $m$. This algebra is $m$-cluster tilted by \cite{Gub} and it is derived discrete by \cite{BGS}. That shows that algebras having different derived categories appear as $m$-cluster tilted algebras.

In subsection \ref{repre finite}, we focus on the representation infinite $m$-cluster tilted algebras arising from a quiver of type $\tAa$. In \cite{FPT} it was proven that every iterated tilted algebra, depending of its global dimension, induces an $m$-cluster tilted algebra which is a higher relation extension, in the sense of \cite{AGS}.  In \cite{FPT}, the authors asked when an $m$-cluster tilted algebra can be realized as a higher relation extension. In this work we prove that representation infinite $m$-cluster tilted algebras arising from a quiver of type $\tAa$, which are non-triangular algebras, can be realized as $m$-relation extensions. More precisely, we prove that a representation infinite $m$-cluster tilted algebras arising from a quiver of type $\tAa$ either is an iterated tilted algebra or an $m$-relation extension of an iterated tilted algebra. 
We have the following theorem.

\begin{teorema*} (Theorem \ref{teorema-extensiones})
Let $A$ be a connected algebra of infinite representation type. The following are equivalent.
\begin{itemize}
\item[(a)] $A$ is an $m$-cluster tilted algebra arising from a quiver of type $\tAa$.
\item[(b)] Is one of the following algebras
\begin{itemize}
\item [(i)] $A$ is an iterated tilted algebra of type $\tAa$, or
\item [(ii)] $A = \mathcal{R}_m(B)=B \semidirprod  \Ext^{m+1}_B(DB,B)$, is the $m$-relation extension of the algebra $B$, where $B = \End_{\Db(H)} (T)$ is an iterated tilted algebra of type $\tAa$ of global dimension $m+1$, with $T$ a tilting complex in the fundamental domain of $\Cc^m (H)$.
\end{itemize}
\end{itemize}
  \end{teorema*}

We observe that this result is not true for the representation finite case. Indeed, if $A$ is a representation finite $m$-cluster tilted algebra arising from a quiver of type $\tAa$, see example \ref{ejemplo 10}, we observe that $A$ is not a trivial extension since it is a triangular algebra and $A$ is not an iterated tilted algebra of any type.

Finally, we give some examples that show the different kinds of algebras that we get as $m$-cluster tilted algebras arising from a quiver of type $\tAa$.

  \section{Preliminaries}\label{section-prelim}

 Consider the derived category $\Db(H)$, where $H=kQ$ is the path algebra of an acyclic quiver $Q$. The \emph{Serre duality} holds in
$\Db(H)$, that is there is a bifunctorial isomorphism
\begin{equation}
\Hom_{\Db(H)}(X, \tau Y)\simeq D \Hom_{\Db(H)}(Y,X[1])
\end{equation}
for all $X,Y \in \Db(H)$, where we denote by $[1]$ the shift functor and by $\tau $ the Auslander-Reiten translation.

The \emph{cluster category} $\Cc(H)$ was defined in \cite{BMRRT}. Later, for a positive integer $m$, the \emph{$m$-cluster category} $\Cc^m(H)$ was defined in \cite{T}. Let $F_m$ be the functor $\tau^{-1}[m]$ defined over $\Db(H)$, then the objects in $\Cc^m(H)$ are the $F_m$-orbits $\tilX$, and the morphisms are given by

\begin{equation}
\Hom_{\Cc^m (H)}(\tilX, \tilY)=\bigoplus_{i\in \mathbb{Z}} \Hom_{\Db (H)}(X,F^i_m Y).
\end{equation}

Given a Krull--Schmidt category $\Cc$, such as $\mmod H$ or $\Cc^m(H)$, we denote by $\ind \Cc$ the set of indecomposable objects over $\Cc$. For an object $X \in \Cc$, we denote by $\add X$ the full subcategory with objects obtained as direct sums of direct summands of $X$. The categories we study have an associated  Auslander--Reiten (AR) quiver, that will be denoted by $\Gamma (\Cc)$. For further information see \cite[I. Section 4.7]{H1}.

The set $\Ss=  \bigcup_{i= 0}^{m-1} \ind \mmod H [i] \bigcup H[m]$ is the \emph{fundamental domain} for the action of $F_m$ in $\Db(H)$.

Let $n$ be the number of non-isomorphic simple $H$-modules, an object $\tilT=\oplus_{i=1}^n \tilT_i$ is \emph{$m$-cluster tilting} if each summand is indecomposable, $\widetilde{T_i} \ncong \widetilde{T_j}$ for $i\neq j$ and $\Ext^i_{\Cc^m(H)} (\tilT,\tilT)=0$ for $1\leq i \leq m$. This definition is not the original due to \cite{T}, is a characterization of $m$-cluster tilting object. See, for instance, \cite{W,ZZ}. Note that, if $\tilT$ is an $m$-cluster tilting object, then  $\add \tilT$ is a cluster tilting subcategory en the sense of \cite{I}. In particular, $\add \tilT$ is functorially finite in $\Cc^m(H)$ and; if for all $X$ in $\Cc^m(H)$ we have that $\Ext^{i}(\tilT', \tilX) = 0$ for all $0< i <m$ and all $\tilT' \in \add \tilT$, then $X$ belongs to $\add \tilT$.

\begin{definicion}
The endomorphism algebra $\End_{\Cc^m(H)}(\tilT)$ of an $m$-cluster tilting object $\tilT$ is called an \emph{$m$-cluster tilted algebra}. We say that such $m$-cluster tilted algebra arises from a quiver of type $Q$ if $H= kQ$.
\end{definicion}

We recall that if $\Aaa$ and $\Bb$ are classes of objects in $\Cc^m(H)$, then $\Aaa \ast \Bb $ denotes the full subcategory of objects $X$ in $\Cc^m(H)$ appearing in a triangle $A \ra X \ra B \ra A[1]$, where $A \in \Aaa$ and $B \in \Bb$.

\begin{notation} \label{remark U}
Consider $\Uu = \add T \ast \add T[1]$. In other words,
$$\Uu = \{\tilX \in \Cc^m(H)/ \mbox{there exists a triangle } \tilT_1 \ra \tilT_0 \ra \tilX \ra \tilde{T_1}[1] \},$$ with $\tilT_0$ and $\tilT_1$ in $\add T$.
\end{notation}

A key result for this work is the following equivalence due to Keller and Reiten, \cite[Section 5.1]{KR}.

\begin{lema} \label{equivalencia} Let $\tilT$ be an $m$-cluster tilting object,
\begin{enumerate}
\item[(a)] For each $\End_{\Cc^m(H)}(\tilT)$-module $M$ there exists a triangle $\tilT_1 \ra \tilT_0\ra \tilX_M \ra \tilT_1 [1] $, with $\tilT_0,\tilT_1 \in \add T$, such that $\Hom_{\Cc^m(H)}(\tilT,\tilX_M)=M$.

\item[(b)] The functor $F=\Hom_{\Cc^m(H)}(\tilT, -)$ induces an equivalence of categories
\begin{equation*}
\Uu /(\add T [1]) \xra{F} \mmod \End_{\Cc^m(H)}(T).
\end{equation*}
 \end{enumerate}

\end{lema}

From now on, let $H$ be a hereditary algebra of Euclidean type.

In this case the AR quiver $\Gamma(\Cc^m(H))$ has $m$ transjective components that we denote $\Cc_0, \ldots, \Cc_{m-1}$ and $m$ sets of regular components denoted by $\Rr_0, \ldots, \Rr_{m-1}$.

\begin{figure}[h!]
\centering
\def\svgwidth{3.3in}
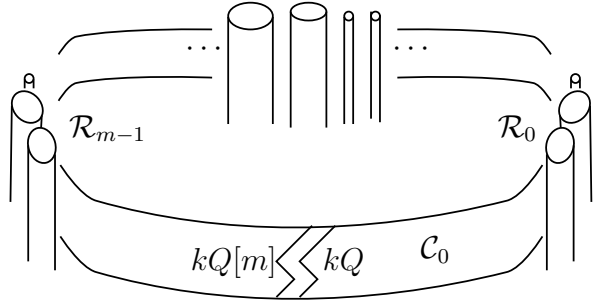
\caption{Auslander--Reiten quiver $\Gamma(\Cc^m(H))$}
\label{categoriamcluster}
\end{figure}

For the proof of our results, will be important to have a good understanding on hereditary Euclidean algebras and their module categories, see \cite{R}.

\section{$m$-cluster tilted algebras}

\subsection{The type $Q$ is not preserved for $m$-cluster tilted algebras.}

In this section we remark that the type $Q$ of an $m$-cluster tilted is not well defined for $m \geq 2$. Gubitosi, in \cite{Gub}, show examples of $m$-cluster tilted algebras  arising from $m$-clusters categories of type $\Aa$ and $\tAa$ at the same time. We illustrate the situation with the following example.

\begin{ejemplo}\label{Ejemplo motivacion} A representation finite $m$-cluster tilted algebra at the same time arising from an $m$-cluster category of type $\tAa$ and from an $m$-cluster category of type $\Aa$.

Let $H$ be the hereditary algebra of type ${\tilde{\mathbb{A}}}_{2,2}$

\begin{center}
\begin{tikzcd}[column sep=small, row sep=small]
&& 2\arrow[rrd]\\
1 \arrow[rru]\arrow[rrd] && && 4\\
&& 3 \arrow[rru]
\end{tikzcd}
\end{center}

Consider the derived category $\Db(H)$  and $\tilT$ a $2$-cluster tilting object in the $2$-cluster category $\Cc^2(H)$, where $T= T_1 \oplus T_2 \oplus T_3 \oplus T_4$ is a representative of $\tilT$ in the fundamental domain $\Ss$ of $\Cc^2(H)$.
The indecomposable summnads are $ T_1= \begin{smallmatrix}&1\\&2\\&4\end{smallmatrix}$  is a regular module in the wide tube $\Tt_1$ of rank $2$, $T_2 = I_4$ is the injective module in the vertex $4$, $T_3 = 2[1]$, where $2$ is the simple regular module in the other wide tube $\Tt_2$  of rank $2$ shifted by one, and $T_4 = I_2[1]$ is the injective module in the vertex $2$ shifted by one.

\begin{figure}[h!]
\centering
\def\svgwidth{5.4in}
\begingroup%
  \makeatletter%
  \providecommand\color[2][]{%
    \errmessage{(Inkscape) Color is used for the text in Inkscape, but the package 'color.sty' is not loaded}%
    \renewcommand\color[2][]{}%
  }%
  \providecommand\transparent[1]{%
    \errmessage{(Inkscape) Transparency is used (non-zero) for the text in Inkscape, but the package 'transparent.sty' is not loaded}%
    \renewcommand\transparent[1]{}%
  }%
  \providecommand\rotatebox[2]{#2}%
  \newcommand*\fsize{\dimexpr\f@size pt\relax}%
  \newcommand*\lineheight[1]{\fontsize{\fsize}{#1\fsize}\selectfont}%
  \ifx\svgwidth\undefined%
    \setlength{\unitlength}{543.74986602bp}%
    \ifx\svgscale\undefined%
      \relax%
    \else%
      \setlength{\unitlength}{\unitlength * \real{\svgscale}}%
    \fi%
  \else%
    \setlength{\unitlength}{\svgwidth}%
  \fi%
  \global\let\svgwidth\undefined%
  \global\let\svgscale\undefined%
  \makeatother%
  \begin{picture}(1,0.17862183)%
    \lineheight{1}%
    \setlength\tabcolsep{0pt}%
    \put(0,0){\includegraphics[width=\unitlength,page=1]{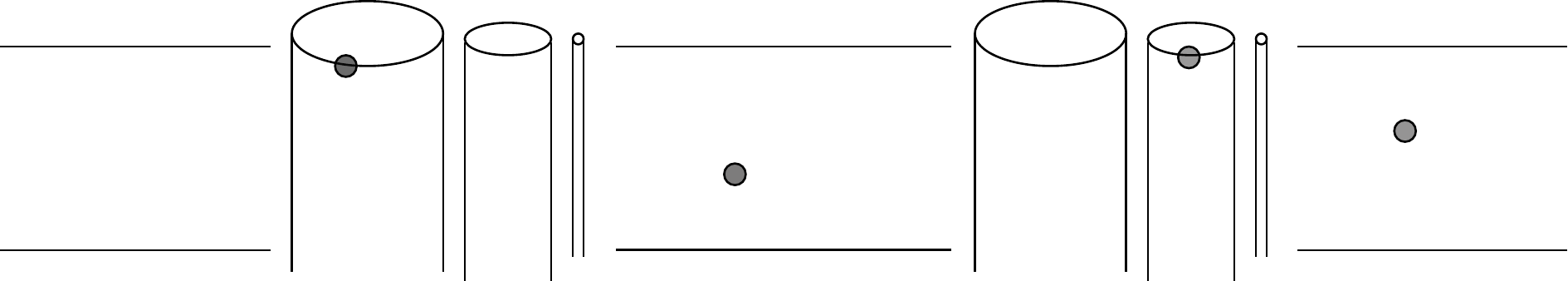}}%
    \put(0.21379321,0.102069){\color[rgb]{0,0,0}\makebox(0,0)[lt]{\lineheight{1.25}\smash{\begin{tabular}[t]{l}$T_1$\end{tabular}}}}%
    \put(0.45655187,0.08275865){\color[rgb]{0,0,0}\makebox(0,0)[lt]{\lineheight{1.25}\smash{\begin{tabular}[t]{l}$T_2$\end{tabular}}}}%
    \put(0.7517244,0.10758623){\color[rgb]{0,0,0}\makebox(0,0)[lt]{\lineheight{1.25}\smash{\begin{tabular}[t]{l}$T_3$\end{tabular}}}}%
    \put(0.90896569,0.07448278){\color[rgb]{0,0,0}\makebox(0,0)[lt]{\lineheight{1.25}\smash{\begin{tabular}[t]{l}$T_4$\end{tabular}}}}%
  \end{picture}%
\endgroup%

\caption{$2$-cluster tilting object in $\Cc^2 (H)$, $H$ of type $\tAa$.}
\label{figura EJ4}
\end{figure}

The endomorphism algebra $A = \End_{\Cc^2(H)}(\tilT )$ is a $2$-cluster tilted algebra arising from the $2$-cluster category of type ${\tilde{\mathbb{A}}}_{2,2}$. The algebra $A=kQ/I$ also is a $2$-cluster tilted algebra arising from a $2$-cluster category of type $\Aa_4$. See Figure \ref{figura ejA4}.

\begin{figure}[h!]
\centering
\def\svgwidth{4.6in}
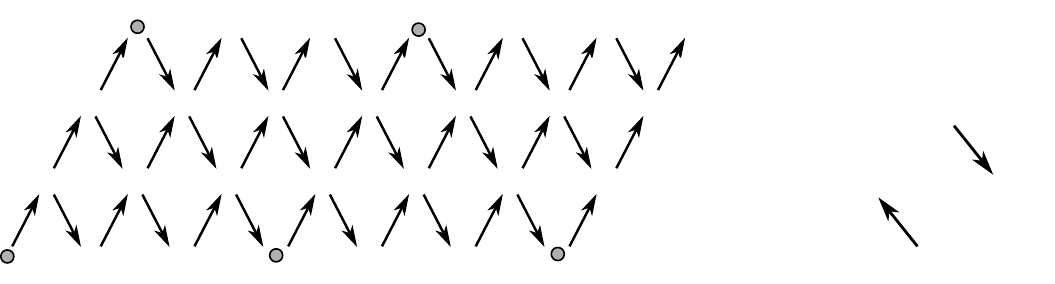
\caption{$2$-cluster tilting object in $\Cc^2 (H)$, $H$ of type $\Aa$. Example \ref{Ejemplo motivacion}.}
\label{figura ejA4}
\end{figure}

\end{ejemplo}

We observe that the algebra $A$ in the previous example is of finite representation type. That shows that the representation type is not preserved, as was the case for cluster tilted algebras. Notice that the $m$-cluster tilted object has direct summand in at least to different transjective components in the the cluster category. We will show that this is the general situation to get finite representation type.

\subsection{Iterated tilted algebras are $m$-cluster tilted algebras}

In this section we show that every iterated tilted algebra is an $m$-cluster tilted algebra for some $m$.

Consider $\Db(H)$ the bounded derived category of a hereditary algebra and let $T$ be \emph{tilting complex}, a complex $T$ such that $\Hom_{\Db(H)}(T,T[j]) = 0$ for all $j \not= 0$ and $\Hom_{\Db(H)}(T,X[j]) = 0 $ for all $j$ implies that $X= 0$. We recall that an algebra $A$ is \emph{iterated tilted} if $A = \End_{\Db(H)} (T)$, that is $A$ is derived equivalent to $H$. In this case both derived categories are triangle equivalent.

\begin{proposicion}
Let $A$ be an iterated tilted algebra  of type $Q$, then there exist $m \in \mathbb{N}$ such that $A$ is $r$-cluster tilted arising from $Q$, for all $r \geq m$.
\end{proposicion}

\begin{proof}
Since $A$ is iterated tilted, then $A = \End_{\Db(H)}(T)$ with $T$ a tilting complex. Suppose that $T$ is spread in  $l$ copies of $\mathrm{ind} \hspace{1pt} H$. Let be $m = l + 2$ and  $r$ greater or equal than $m$, consider the $r$-cluster category, then we have  $\Hom_{\Db(H)}(T,Z[r]) = 0$ if $Z$ belongs to $\mathrm{ind} \hspace{1pt} H$, since $r \geq l+2$ and $\Db(H)$ is the derived category of a hereditary algebra . It follows that $\End_{\Cc^r(H)}(\tilT) = A$, then $A$ is $r$-cluster tilted algebra, for all $r \geq m$.\end{proof}

\begin{remark}
Recently, Ladkani in \cite{L} proved that every finite dimensional algebra is $m$-Calabi-Yau tilted for some $m$ greater or equal than three.  It is important to recall that $m$-cluster categories are $(m+1)$-Calabi-Yau categories, so $m$-cluster tilted algebras are $(m+1)$-Calabi-Yau tilted algebras. This result is a particular case of Ladkani's result.
\end{remark}

\section{Euclidean $m$-cluster tilted algebras of finite representation type}

In this section we will characterize representation finite $m$-cluster tilted algebras arising from a hereditary algebra of Euclidean type. We will show that we can get this characterization in terms of the direct summands of the $m$-cluster tilting object. We will prove that this is the case if the $m$-cluster tilting object has direct summands in at least two different transjective components in the $m$-cluster category. The proof of the main theorem is inspired in the principal result of \cite{AHT} on iterated tilted algebras. Since iterated tilted algebras are $m$-cluster tilted algebras is not surprising that a similar result holds for both classes of algebras.

We start recalling the definition of right $\add M$-approximation. Given an $A$-module $M$, we say that a morphism $f_0: M_0 \ra X$, with $M_0$ in  $\add M$, is a right  $\add M$-\emph{approximation} if for every morphism $f_1: M_1 \ra X$ with $M_1$ in $\add M$, there exists a morphism $g: M_1\ra M_0$ such that $f_0g= f_1$. Equivalently, $f_0: M_0 \ra X$ with $M_0$ in  $\add M$ is a right $\add M$-approximation if $\Hom_A(-,f_0): \Hom_A(-,M_0) \ra \Hom_A(-,X)$ is surjective in $\add M$.

\smallskip

We prove a useful lemma.

\begin{lema}\label{isomorfismo}
Let $\tilde{T}$ be an $m$-cluster tilting object and $\tilX $ in $\Cc^m(H)$. Then there exist a triangle
\[\tilde{C} \ra \tilT_0 \stackrel{f}\ra \tilX \ra \tilde{C}[1],\]
such that $f$ is a right $\add \tilT$-approximation.
Moreover, we have that \[\Hom_{\Cc^m(H)}(\tilT,\tilde{C}[1])= 0 \hspace{6pt} \] and \[ \Hom_{\Cc^m(H)}(\tilT,\tilX[i]) \cong \Hom_{\Cc^m(H)}(\tilT,\tilde{C}[i+1]) \hspace{5pt} \mathrm{for \ all} \hspace{5pt} i=1, \ldots,m-2. \]
\end{lema}
\begin{proof}

Since $\add \tilT$ is an $m$-cluster tilting subcategory, then is functorially finite. It follows that for all $\tilX$ in $\Cc^m(H)$ there exists $f: \tilT_0 \ra \tilX$ a right $\add \tilT$-approximation of $\tilX$. Then there exists a triangle
\[\tilde{C} \ra \tilT_0 \stackrel{f}\ra \tilX \ra \tilde{C}[1].\]
We apply $\Hom_{\Cc^m(H)}(\tilT, -)$ and we get a long exact sequence
\[ \cdots \ra \Hom_{\Cc^m(H)}(\tilT, \tilde{C}) \ra \Hom_{\Cc^m(H)}(\tilT,\tilT_0) \ra \Hom_{\Cc^m(H)}(\tilT, \tilX) \ra \Hom_{\Cc^m(H)}(\tilT, \tilde{C}[1]) \]
\[ \ra \Hom_{\Cc^m(H)}(\tilT,\tilde{T_0}[1]) \ra \Hom_{\Cc^m(H)}(\tilT, \tilde{X}[1]) \ra \Hom_{\Cc^m(H)}(\tilT,\tilde{C}[2]) \ra \cdots\]

Since $\Hom_{\Cc^m(H)}(\tilT, \tilde{T_0}[i]) = 0$ for all $i = 1, \ldots, m-1$, we obtain the isomorphism  \[ \Hom_{\Cc^m(H)}(\tilT, \tilde{X}[i]) \cong \Hom_{\Cc^m(H)}(\tilT, \tilde{C}[i+1])\] for $ 1\leq i \leq m-2$. Since $f$ is an $\add \tilT$-approximation, we get that $\Hom_{\Cc^m(H)}(\tilT, \tilde{C}[1]) = 0$.
\end{proof}

There exists a nice connection between $m$-cluster tilting objects in the $m$-cluster category and silting complexes in the fundamental domain in the derived category of the hereditary algebra.

We recall the following definitions.

A basic object $T$ in $\Db(H)$ is said to be \emph{partial silting} if $\Hom(T, T[i]) = 0$ for $i > 0$ , and \emph{silting} if, in addition, $T$ is maximal with this property. In \cite[Theorem 2.3]{AST}, the authors prove that the number of indecomposable summands of $T$ is bounded by the number of simple modules of the algebra $H$ and it is a \emph{silting} if the number of indecomposable summands is equal to the number of simple modules. In \cite[Proposition 2.4]{BRT}, the authors show that if $T$ is an object in the fundamental domain $\Ss$ of $\Db(H)$, then $\tilT$ is an $m$-cluster tilting object in $\Cc^m(H)$ if and only if $T$ is a silting complex in $\Db(H)$.

In the next lemma, we show that there is always a transjective direct summand for $m$-cluster tilting objects in the $m$-cluster category of a hereditary algebra of Euclidean type.

\begin{lema} Let $\Cc^m(H)$ be an $m$-cluster category of Euclidean type. Let $\tilT$ be an $m$-cluster tilting object. Then, $\tilT$ has at least a transjective summand.
\end{lema}

\begin{proof}
Without loss of generality we can choose a representative $T$ in the fundamental domain $\Ss$ of $\Db(H)$. By \cite[Proposition 2.4]{BRT}, $T$ is a silting complex in the fundamental domain $\Ss$ of $\Db(H)$. By \cite{AST} Lemma 3.1 (a) and Corollary 4.4, we have that $T$ must be a generator of $\Db(H)$, so the smallest triangulated category generated by $T$ must be $\Db(H)$.  If all the direct summands of $T$ belong to regular components, then $\underline{\dim} \hspace{1pt} T_i$ are linearly dependent, a contradiction to the fact that they must form a basis of the Grothendieck group of $\Db(H)$, by \cite[III.1.2]{H1}.
\end{proof}

The following remark will be very useful for the proof of our mains results.

\begin{remark} \label{minimalrepinf} Let $H$ be a hereditary algebra of Euclidean type and let $e$ be an idempotent element in $H$. Then the quotient algebra $H/\langle e \rangle$ is of finite representation type. In particular, if $P$ is a non-zero projective $H$-module then $ \Hom_H (P,M) \neq 0$ for almost all $M$ in $\mod H$. Dually, if $I$ is a non-zero injective module, then $\Hom_H(N,I)\neq 0$ for almost all $N$ in $\mod H$.
\end{remark}

We start proving our first result.

\begin{proposicion} \label{repfinite}
Let $A = \End_{\Cc^m(H)}(\tilT)$ be an $m$-cluster tilted algebra arising from a hereditary algebra $H$ of Euclidean type such that the $m$-cluster tilting object $\tilT$ has at least two direct summands in different transjective components of $\Gamma(\Cc^m (H))$. Then $A$ is of finite representation type.
\end{proposicion}
\begin{proof}
We claim that for almost all $X \in \Ss$, we have that $\Hom_{\Cc^m(H)}(\tilT,\tilde{X}[i]) \not= 0$ for some $i = 1, \cdots m-2$.

In fact, let $T$ be a representative of $\tilT$ in the fundamental domain $\Ss$. Without lost of generality, we can assume that the first transjective direct summand $T_0$ of $T$ is in the transjective component $\mathcal{C}_0$. Since $T$ has at least two transjective direct summands in different components. Suppose that other transjective direct summand $T_i$ lie in $\mathcal{C}_{k+1}$. We can assume without loss of generality that $T_i$ and $\tau T_i$ lie in $\mmod H [k]$, for $0< k \leq m-2$, changing the algebra $H$ if necessary.

We study the case of all the indecomposable objects in the fundamental domain $\Ss$.

We first consider $X = Z[j]$ in $\Ss$, with $Z$ in $\mmod H$ and $j = 0, \dots, k$. We can suppose that $\tau T_i$ is an injective $H'$-module over a hereditary algebra $H'$, choosing $H'$ the hereditary algebra obtained from a complete slice passing through $T_i$. Let be $r=k-j$, then by the Serre duality,  $$\Hom_{\Db(H)}(X[r], \tau T_i) = D\Hom_{\Db(H)}( T_i, X[r+1]).$$
We have that $\Hom_{\Db(H)}(X[r], \tau T_i) = \Hom_{\Db(H)}(Z[k], \tau T_i) \not = 0$ for almost all $Z$, by Remark \ref{minimalrepinf}. We obtain that $\Hom_{\Db(H)}( T_i, X[r]) \not= 0$ for some $r = 0, \cdots m-3$.

Consider, as in \ref{isomorfismo}, the triangle \[\tilde{C} \ra \tilT_0 \stackrel{f}\ra \tilX \ra \tilde{C}[1],\]

We observe, since $\add \tilT$ is a cluster tilting subcategory, that $\widetilde{C}$ belongs to $\add \tilT$ if and only if $\Hom_{\Cc^m(H)}(\tilT,\tilde{C}[l]) \cong 0$ for all $l = 1, \cdots, m-1$.

By Lemma \ref{isomorfismo}, we have that $\Hom_{\Cc^m(H)}(\tilT,\tilde{C}[l+1])= \Hom_{\Cc^m(H)}(\tilT,\tilde{X}[l])\not  = 0$ for some $l$ in $\{ 1, \ldots, m-2\}$.

\smallskip

Secondly, if $X= Z[j]$ with $j = k+1, \cdots, m-1$. Le be $r=m-1-j$, that is $r= 0, \cdots, m-k-2$. Then, we have
\[\Hom_{\Cc^m(H)}(\tilde{X}[r], \tau \tilde{T}_0) = \Hom_{\Cc^m(H)}(\tilde{Z}[m-1], \tau \tilde{T}_0) \not = 
0\] for almost all $Z$, since $\Hom_{\Cc^m(H)}(\tilde{Z}[m-1], \tau \tilde{T}_0) \cong \Hom_{\Db(H)}(\tilde{Z}[m-1], \tau \tilde{T}_0)$ by Remark \ref{minimalrepinf}.

It follows that $\Hom_{\Cc^m(H)}(\tilde{X}[r], \tau \tilde{T}_0) = D\Hom_{\Cc^m(H)}(\tilde{T}_0, \tilde{X}[r+1]) \not = 0$ for some $r= 0,\cdots, m-3$, then $\Hom_{\Cc^m(H)}(\tilde{T}_0, \tilde{X}[r+1]) = \Hom_{\Cc^m(H)}(\tilde{T}_0, \tilde{C}[r+2]) \neq 0$ for some $r+2 = 2, \cdots,  m-1$. We get that $\tilde{C}$ does not belong to $\add \tilde{T}$, then $X$ is not in $\Uu$ (see Remark \ref{remark U}).

\smallskip

We have proven that for almost all $X$ in  $\Ss$ we obtain that $\tilde{X}$ is not in $\Uu$. Using, Lemma \ref{equivalencia}, we get that $A$ is of finite representation type.
\end{proof}

In the next step, we show that the converse is also true.

\begin{proposicion} \label{repinfinite} Let $\tilT$ be an $m$-cluster tilting object such that all its transjective direct summands lie in a single transjective component $\Cc_0$. Let $\Tt_\lambda$ be a homogeneous tube in $\Rr_0$. Then, almost all $\tilX \in \Tt_\lambda$ lie in $\Uu$, hence $\End_{\Cc^m(H)} (\tilT) = A$ is a representation infinite algebra.
\end{proposicion}

\begin{proof}
Let $\tilT_k$ be the transjective direct summand of $\tilT$, in such way that $ \tilT = \tilT_k \oplus \tilT_l$ where $\tilT_l$ is a regular direct summand.  Then we have, by Remark \ref{minimalrepinf}, that $\Hom_{\Db(H)}(T_k,X) \not= 0$ for almost all $X$ in the homogeneous tube $\Tt_\lambda$. Let $f_0: \tilT_0 \ra \tilX$ the $\add \tilT$-approximation  of $\tilX$.

By Lemma \ref{isomorfismo}, we have that $\Hom_{\Cc^m(H)}(\tilT,\tilde{C}[1])=0$ and \[\Hom_{\Cc^m(H)}(\tilT,\tilX[i]) \cong \Hom_{\Cc^m(H)}(\tilT,\tilde{C}[i+1]) \hspace{5pt} \mathrm{ for } \hspace{5pt} i=1 ,\ldots, m-2.\]

We claim that $\Hom_{\Cc^m(H)}(\tilT,\tilX[i])= 0$ for all $i$ from $1$ to $m-1$. In fact, let $T$ be a representative of $\tilT$ in $\Ss$.

By \cite[Lemma 2.1]{FPT}, we have that\[\Hom_{\Cc^m(H)}(\tilT,\tilX) = \Hom_{\Db(H)}(T,X) \oplus \Hom_{\Db(H)}(T,F^{-1}X) \hspace{5pt} \mathrm{for} \hspace{5pt} T \hspace{5pt}\mathrm{and} \hspace{5pt} X \hspace{5pt} \mathrm{in} \hspace{5pt} \Ss .\]

It holds that $\Hom_{\Db(H)}(T_l,X[i]) = D\Hom_{\Db(H)}(X[i-1],\tau T_l) = 0$ for all $i > 0$, since $X$ is homogeneous, and $T_l$ belongs to a wide tube and there are not morphism between them.

On the other hand, $\Hom_{\Db(H)}(T_k,F^{-1}X[i]) = \Hom_{\Db(H)}(T ,\tau X[i-m]) = 0$ since $i-m <0$ and there are not morphisms going back in the derived category of a hereditary algebra. Then $\Hom_{\Db(H)}(T,X[i])$  is zero
for all $i > 0$.

The same argument shows that there are not non zero morphisms from $T_l$ to $X[i-m]$. Thus, we have $\Hom_{\Cc^m(H)}(\tilT,\tilX[i])= 0$ for all $1 \leq i \leq m-1$.

By Lemma \ref{isomorfismo}, $\Hom_{\Cc^m(H)}(\tilT,\tilde{C}[i+1]) = 0$ for all $i$ from $1$ to $m-2$. Then we have $\Hom_{\Cc^m(H)}(\tilT,\tilde{C}[i]) = 0$, for all $i= 1, \cdots, m-1$.  Then $\tilde{C} \in \add \tilde{T}$ and
$\tilX$ belongs to $\Uu$ for almost all $X$ in $\Tt_\lambda$. By Lemma \ref{equivalencia}, the module category $\mmod A$  has infinitely many indecomposables and $A$ is representation infinite.
\end{proof}

Combining the results in Proposition \ref{repfinite} and Proposition \ref{repinfinite}, we can state the main theorem of the section.

\begin{teorema}\label{theorem-main}
Let $\Cc^m(H)$ be the $m$-cluster category of Euclidean type and let $A= \End_{\Cc^m(H)}(\tilT)$ be an $m$-cluster tilted algebra. Then, $A$ is of finite representation type if and only if the $m$-cluster tilting object $\tilT$ has summands in at least two different transjective components of $\Gamma(\Cc^m (H))$.
\end{teorema}


\section{The case $\tAa$.}\label{section A tilde}

\subsection{Geometric realization}\label{geomteric realization}
When the underlying graph $\overline{Q}$ is of type $\Aa$ or $\tAa$, the $m$-cluster categories and their corresponding $m$-cluster tilting objects were realized geometrically and studied in \cite{Ba,Mu}, and \cite{Tol,Gub}, respectively. These geometric realizations generalize those from \cite{CCS,ABCP} in the case of the unpunctured disc and the annulus.

\vspace{4pt}

\textit{\bf The case $\Aa$}: Let $\Pi$ be a disk with $nm+2$ marked points (or equivalently an $(nm+2)$-gon). A \emph{$m$-diagonal} is a diagonal dividing $\Pi$ into an $(mj+2)$-gon and an $(m(n-j) + 2)$-gon for some $ 1 \leq j \leq n-1/2 $. A $(m+2)$-angulation is a collection of non-intersecting $m$-diagonals that form a partition of $\Pi$ into $(m+2)$-gons. There are bijections:

\begin{center}
\begin{tabular}{ccc}
$m$-diagonals in $\Pi$ &$\leftrightarrow$ &indecomposable objects in $\Cc^m_{\Aa}$
\\
$(m+2)$-angulations of $\Pi$ &$\leftrightarrow$& $m$-cluster tilting objects in $\Cc^m_{\Aa}$

\end{tabular}
\end{center}

\textit{ \bf The case $\tAa$}: While in \cite{Tol,Gub} the authors use an annulus with marked points, we will use the universal cover given by the strip $\Sigma$, having copies of the $mp$, and $mq$, marked points on the boundary components that we denote $\mathcal{B}_p$, and $\mathcal{B}_q$, respectivelly. The points having the same label in a fixed boundary are considered up to equivalence given by congruence modulo $mp$ and $mq$. There are (isotopy classes of) arcs on $\Sigma$, called \emph{$m$-diagonals}. Each $m$-diagonal belongs to a family:
\begin{enumerate}
\item[$\ast$] \emph{Transjective}: an arc $\alpha$ having an endpoint $x$ in $\Bp$ and the other endpoint $y$ in $\Bq$. The labels in $x$ and $y$ are congruent modulo $m$.
\item[$\ast$] \emph{Regular on a $p$-tube}: an arc $\alpha$ having both endpoints over $\mathcal{B}_p$, starting at $u$ going in positive direction counting $u+km+1$ steps, with $k\geq 1$.
\item[$\ast$] \emph{Regular on a $q$-tube}: analogous to the previous case.
\end{enumerate}

As in the previous case, there are bijections

\begin{center}
\begin{tabular}{ccc}
$m$-diagonals in $\Sigma$ &$\leftrightarrow$ &indecomposable rigid objects in $\Cc^m_{\tAa}$
\\
$(m+2)$-angulations of $\Sigma$ &$\leftrightarrow$& $m$-cluster tilting objects in $\Cc^m_{\tAa}$

\end{tabular}
\end{center}

\begin{remark}  For a transjective arc, the common congruence class $u$ modulo $m$ in the endpoints labels $x,y$ indicates that the associated object sits in the $u$-th transjective component of $\Gamma(\Cc_Q^m)$. Also, the endpoints of the regular arcs mark the tranjective component where the associated object lies. See \cite[Section 2.3.2]{Gub}.
\end{remark}

Given an $(m+2)$-angulation $\Tt$ of $\Pi$ or $\Sigma$, the bound quiver $(Q_\Tt,I_\Tt)$ of the $m$-cluster tilted algebra $\Lambda$ defined by the associated $m$-cluster tilting object is obtained form the geometric configuration, see \cite{Gub,Mu}. We recall the geometric realizations of $\Cc^m_Q$ and the bound quiver construction in the following example.

\begin{figure}[h!]
\centering
\def\svgwidth{5.3in}
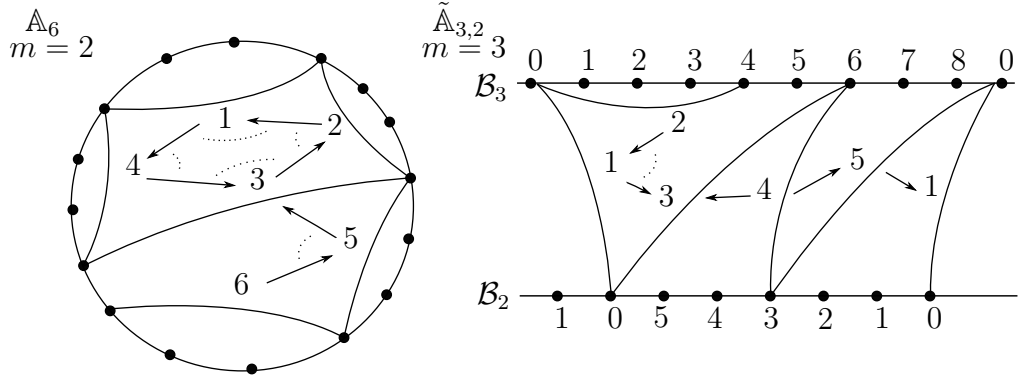
\caption{$4$-angulation of $\Pi$ (left) and $5$-angulation of $\Sigma$ (right), and bound quivers defined by them. The labels at the endpoints of transjective arcs are congruent} 
\label{ejemplo2}
\end{figure}

\begin{ejemplo}\label{ejemplo1}  In Figure \ref{ejemplo2}, we show the bound quiver defined by a $4$-angulation of $\Pi$ (left) that corresponds to a $2$-cluster tilting object in $\Cc^m_{\Aa_6}$, and a $5$-angulation of $\Sigma$ (right) that corresponds to a $3$-cluster tilting object in $\Cc^m_{\tAa}$ where $p=3$ and $q=2$. In both cases,
the vertices in $Q_\Tt$ are in one-to-one correspondence with the elements in $\Tt$. For any two vertices $i,j \in Q_\Tt$, there is an arrow $i \ra j$ when the corresponding $m$-diagonals $x_i$ and $x_j$ share a vertex, they are edges of the same $(m+2)$-gon and $x_i$ follows $x_j$ clockwise. Given consecutive arrows $i \xrightarrow{\alpha} j \xrightarrow{\beta} k$, then $\alpha \beta \in I_\Tt$ if and only if $x_i$, $x_j$ and $x_k$ are edges in the same $(m+2)$-gon. Observe that all the transjective $m$-diagonals in the $(m+2)$-angulation of $\Sigma$ are associated to objects in the $0$-th transjective component of $\Gamma(\Cc^m_Q)$.

\end{ejemplo}

\begin{definicion}\label{saturated} We say that a cycle $a_1 \xra{\alpha_1} a_2 \to \cdots \to a_n \xra{\alpha_n} a_1 $ is \emph{saturated} if $\alpha_i \alpha_{i+1} \in I_\Tt$ for all $i$ where the index runs over all congruence classes modulo $n$.
\end{definicion}

\begin{definicion}\label{root cycle}
Let $(Q_\Tt, I_\Tt)$ be a bound quiver defined from an $(m+2)$-angulation of the strip $\Sigma$. We say that $(Q_\Tt,I_\Tt)$ has a \emph{root cycle} $\mathfrak{C}$ if $Q_\Tt$ contains a closed reduced walk that is not a saturated cycle.

\end{definicion}

It is simple to verify that a bound quiver $(Q_\Tt,I_\Tt)$ arising form a $(m+2)$-angulation of $\Sigma$ has the following properties:

\begin{enumerate}
\item[i)] $(Q_\Tt,I_\Tt)$ is gentle.
\item[ii)] If $Q_\Tt$ has a root cycle with no relations, then it is the only non-oriented closed reduced walk.
\end{enumerate}

In Example \ref{ejemplo1}, we have a root cycle given by $1\ra 3 \la 4 \ra 5 \la 1$. The representation type of gentle (and more generally string) algebras can be determined in terms of bound quivers, by a result from \cite{BuR}. A gentle algebra $\Lambda=kQ/I$ is always tame, and it is of finite representation type if and only if $(Q,I)$ does not have a closed reduced walk $w$ in which no subpath belongs to $I$.

\begin{lema}\label{lema trans}
Let $\Tt$ be an $(m+2)$-angulation of $\Sigma$, and $m \geq 2$. Let $\Pi_0$ be an $(m+2)$-gon with edges $b_u$ and $b_w$, such that both vertices of $b_u$ lie in $\Bp$, respec. $\Bq$. Then, $\Pi_0$ has two transjective $m$-diagonals $x_1, x_2$ as edges, and the associated objects lie in different transjective components of $\Gamma (\Cc^m_{\tAa})$.

\end{lema}

\begin{proof}
Let $\Pi_0$ be as above, we illustrate the general situation in Figure \ref{figura lema}. In order to have a $(m+2)$-gon we need two edges $x_1$ and $x_2$. We will see that they are different.

\begin{figure}[h!]
\centering
\def\svgwidth{4in}
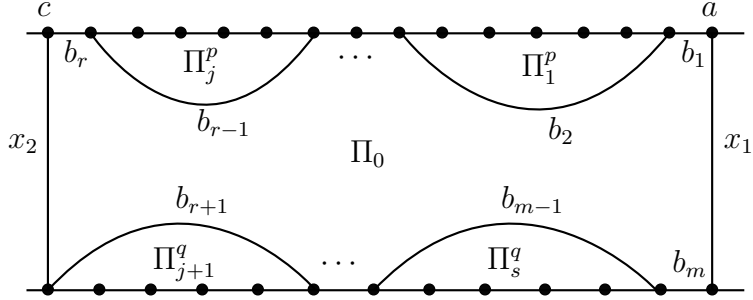
\caption{Lema \ref{lema trans}. Such configuration does not allow the labels $a$ and $c$ to be congruent modulo $m$.}
\label{figura lema}
\end{figure}

The other edges of $\Pi_0$ are either boundary edges (as $b_1$ on Figure \ref{figura lema}) or regular $m$-diagonals (as it is $b_2$). Let $E^p =\{ b_1, \ldots, b_r \}$ be the edges having vertices on $\Bp$ and $E^q= \{ b_{r+1}, \ldots, b_{m} \}$ be the edges having vertices on $\Bq$. By hypothesis the sets $E^p, E^q$ are non empty, adding the information that $\Pi_0$ has $m+2$ edges, we conclude that each set has cardinality at most $m-1$.
If $b_i$ is an $m$-diagonal, then it encloses a polygon $\Pi^p_t$, and since $b_i$ is part of an $(m+2)$-angulation $\Pi^p_t$ is force to have $m n_i +2$ edges ($n_i$ is a positive integer). Then $\Pi^p_t$ has $m n_i +1$ edges that are segments lying on the boundary. Let $a$ be the label in the vertex of $x_1$ lying on $\Bp$. Now we count the number of boundary segments from $a$ to $c$. After $b_1$, we add one to $a$. When we reach the right endpoint of $b_2$ we add $m n_i +1$, that is we add $1$ modulo $m$. When we reach $c$, we would have added $r$ modulo $m$, but $r= \vert E^p \vert $ and $1 \leq \vert E^p \vert  \leq m-1$, so the labels $c$ and $a$ are not congruent modulo $m$. Therefore the objects associated to $m$-diagonals $x_1$ and $x_2$ are in different transjective components.\end{proof}

It was already observed in \cite[Proposition 4.3]{Tol} that given an $(m+2)$-angulation $\Tt$ of $\Sigma$, there is at least one transjective $m$-diagonal $x \in \Tt$. 

In the case $m \geq 2$, the previous lemma shows that when the $(m+2)$-angulation $\Tt$ has an $(m+2)$-gon with two edges one with endpoints in $\Bp$, and the other one in $\Bq$, then $\Tt$ is forced to have at least two transjective $m$-diagonals. If there is no $(m+2)$-gon having this property in $\Tt$, then the transjective $m$-diagonals form zig-zags and fans, see Figure \ref{figura sec5}. Those zig-zag and fans will have at least two transjective $m$-diagonals. In either case we observe that for $m \geq 2$ there is at least two transjective $m$-diagonals in every $(m+2)$-angulation.

\begin{teorema}\label{teorema A}
Let $\Lambda_\Tt = \End_{\Cc^m} (T)$ be an $m$-cluster-tilted algebra arising from an $m$-cluster category of type $\tAa$. Then,

\begin{enumerate}
\item[(a)] $\Lambda_\Tt$ arises at the same time from an $m$-cluster category of type $\Aa$ if and only if $Q_\Tt$ does not have a root cycle.

\item[(b)] $\Lambda_\Tt$ is of infinite representation type if and only if all transjective summands of $T$ lie in the same AR component.
\end{enumerate}

\end{teorema}

\begin{proof}
(a) $Q_\Tt$ does not have a root cycle if and only if the $(m+2)$-angulation has a $(m+2)$-gon $\Pi_0$ with (at least) two edges $b_p, b_q$ such that $b_p \subset \Bp$ and $b_q \subset \Bq$.

\begin{figure}[h!]
\centering
\def\svgwidth{6in}
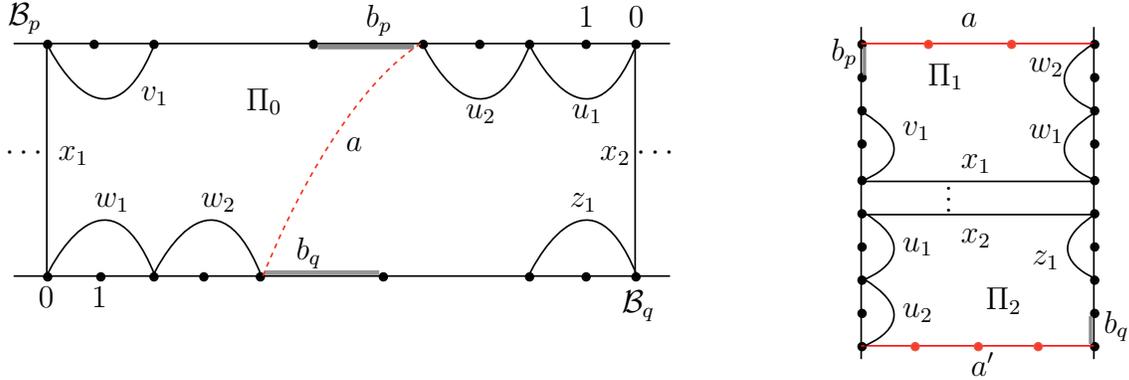
\caption{Defining a $(m+2)$-angulation of a polygon from a $(m+2)$-angulation of $\Sigma$.}
\label{propo2a}
\end{figure}

Denote by $b_p$ the edge of $\Pi_0$ contained in $\Bp$ that has a minimal value in its labels. By the same rule, denote $b_q$ the edge in $\Bq$. By Lemma \ref{lema trans}, the transjective $m$-diagonals $x_1$ and $x_2$ are associated to objects lying in different transjective components. Draw an arc $a$ as in Figure \ref{propo2a}, notice that this arc does not need to be an $m$-diagonal. Cut the strip $\Sigma$ trough $a$, in this way we have a polygon with edges $a$ and $a'$ as we see in Figure \ref{propo2a} (right). The polygon is almost $(m+2)$-angulated, all sub-polygons except $\Pi_1$ and $\Pi_2$ are $(m+2)$-gons. Finally, add marked points to $a$ and $a'$. In this way we obtain an $(nm+2)$-gon $\Pi$, and has an $(m+2)$-angulation obtained from $\Tt$ not adding any other $m$-diagonal or changing their relative positions. Then, $(Q_\Tt,I_\Tt)$ can be obtained as well as the bound quiver arising from a $(m+2)$-angulation of a polygon. Therefore $\Lambda_\Tt = k Q_\Tt/ I_\Tt$ is, also, an $m$-cluster-tilted algebra arising from a $m$-cluster category of type $\Aa$.

(b) For a gentle algebra to be of infinite representation type we need a root cycle with no relations, according to \cite{BuR}. The only way to obtain such a root cycle in  $Q_\Tt$ is that all the transjective $m$-diagonals form fans and zig-zags as in the Figure \ref{figura  sec5}. Hence, for all the endpoints in these $m$-diagonals we have labels all congruent modulo $m$. This means that the corresponding objects are all in the same transjective component.\end{proof}

\begin{figure}[h!]
\centering
\def\svgwidth{4.8in}
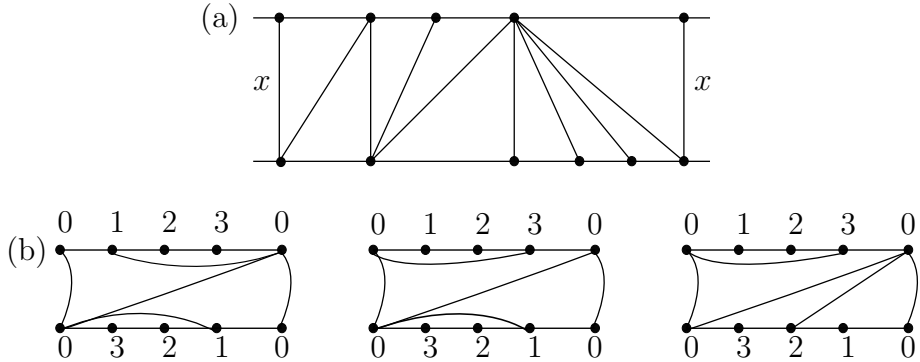
\caption{(a) A root cycle without zero relations in $Q_\Tt$ arises from fans and zig-zags; (b) Cases of summands of $T$ in regular components that lead to representation infinite algebras.}
\label{figura sec5}
\end{figure}

We can have representation infinite $m$-cluster tilted algebras such that $T$ has summands in both regular components that lie in $\mmod H [0]$. An example of this is Figure \ref{figura sec5} (b) left. In Figure \ref{figura sec5} (b) center,  we have an infinite representation case where $T$ has summands in different regular components, one in $\mmod H [0]$ and the other one in $\mmod H[1]$. The remaining example on the right corresponds to $T$ having only one summand in a regular component.

We show now more examples of $m$-cluster tilted algebras arising from $(m+2)$-angulations.

\begin{ejemplo}\label{Ejemplo 7} A representation finite $m$-cluster tilted algebra at the same time arising from an $m$-cluster category of type $\tAa$ and from an $m$-cluster category of type $\Aa$. Figure \ref{figura EJ1} (a).

\begin{ejemplo}\label{ejemplo-derived-discrete} A representation finite $m$-cluster tilted algebra arising from an $m$-cluster category of type $\tAa$ which is derived discrete, see \cite{V} and \cite{BGS}, Figure \ref{figura EJ1} (b).

\end{ejemplo}

\begin{figure}[h!]
\centering
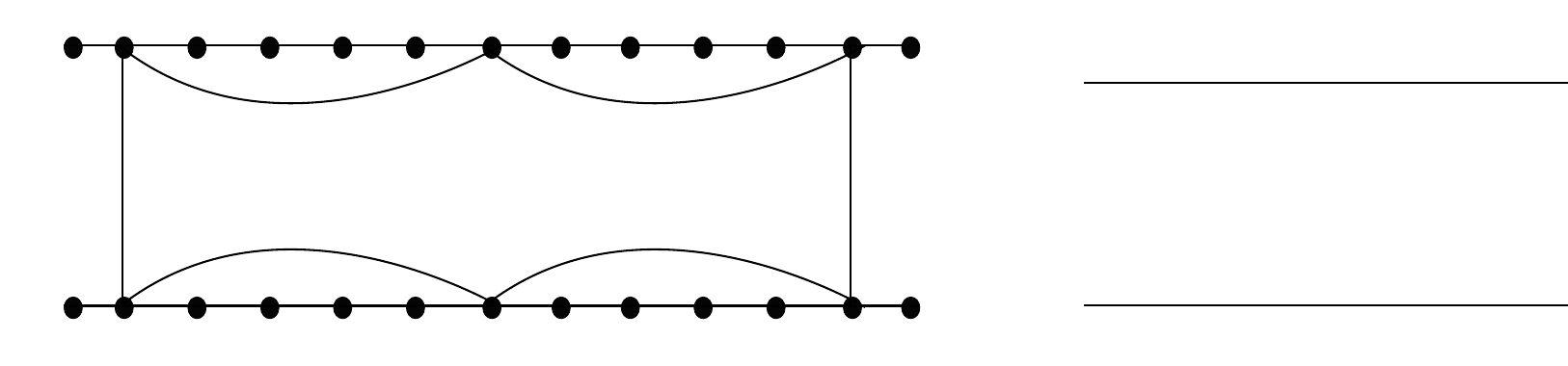
\def\svgwidth{4.2in}
\caption{Examples \ref{Ejemplo 7} and \ref{ejemplo-derived-discrete}}
\label{figura EJ1}
\end{figure}

\end{ejemplo}

 As it was observed in the introduction, a representation finite $m$-cluster tilted algebra arising from type $\tAa$ does not need to be an $m$-cluster tilted algebra of type $\Aa$. Example \ref{ejemplo-derived-discrete} is an example of this.

\begin{ejemplo} \label{repinfiniteAntilde} A representation finite $m$-cluster tilted algebra arising from an $m$-cluster category of type $\tAa$ which is representation infinite.

\begin{figure}[h!]
\def\svgwidth{4.8in}
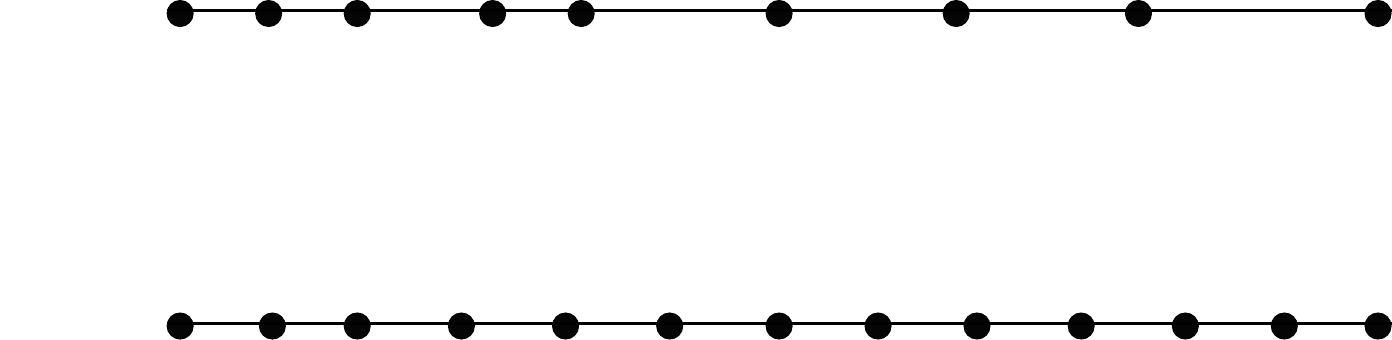
\caption{Example \ref{repinfiniteAntilde}.}
\label{figura EJ5}
\end{figure}

The $(m+2)$-angulation in Figure \ref{figura EJ5} defines the bounded quiver:
\begin{center}
\begin{tikzcd}[column sep=small]
5 \arrow[rr]&& 3  \\
1 \arrow[rr, shift left,"\beta"]
\arrow[rr, shift right,"\alpha_1"']
 && 2 \arrow[u,"\gamma"]\arrow[rr,"\alpha_2"'] && 4\arrow[d,"\alpha_3"]\\
&& 7\arrow[llu,"\alpha_5"] && 6\arrow[ll,"\alpha_4"] \end{tikzcd}
with relations $\alpha_i \alpha_{i+1}=0$ for $i$ integer modulo $6$, and $\beta \gamma=0$.
\end{center}

\end{ejemplo}

\subsection{The representation infinite case - Realization as $m$-relation extensions}\label{repre finite}

In this section we are going to characterize the $m$-cluster tilted algebras of infinite representation type arising
from an $m$-cluster category of type $\tAa$ as being iterated tilted algebras or $m$-relation extensions of iterated tilted algebras of type $\tAa$.

We start recalling some results. We recall from \cite{Gub} that if a connected algebra is an $m$-cluster tilted algebra arising from $\tAa$ then it has a description in terms of quiver with relations. Combining
 this with \cite{BuR}, we get:

 \begin{remark} \label{quiverconrelaciones} A connected algebra $A =KQ/I$ is an $m$-cluster tilted algebra arising from  $Q$ of type $\tAa$ of infinite representation type if and only if $(Q,I)$ is a gentle bound quiver satisfying the following conditions:
 \begin{itemize}
 \item [a)] It contains a unique non-saturated root cycle $\mathfrak{C}$ without zero relations.
 \item[b)] If the quiver contains more cycles, then all of them are saturated cycles.
 \item[c)] Outside of a saturated cycle it can contain at most $m-1$ consecutive relations.
 \end{itemize}
 \end{remark}

 Recall that the trivial extension of $B$ by $M$
is the algebra $B \semidirprod  M$ with underlying vector space $B
\oplus M$, and multiplication given by $(a,m)(a',m') = (aa',am' +
ma')$, for any $a,\;a' \in B$ and $m,\;m' \in M$. For further
properties of trivial extensions we refer the reader to
\cite{FGR}.

Let $B$ be an algebra of global dimension at most
$m+1$ and $DB= {\rm Hom}_k(B,k)$. The trivial extension
$\mathcal{R}_m(B)=B \semidirprod  \Ext^{m+1}_B(DB,B)$ is called
the $m$-relation extension of $B$.

This concept is the $m$\emph{-ified}
analogue of the notion of relation extension introduced in
\cite{ABS}. In this work the authors proved that an algebra ${C}$
is a cluster tilted algebra if and only if it is the relation
extension of some tilted algebra $B$.

We are now in a position to state our result.

\begin{teorema} \label{teorema-extensiones}
Let $A$ be a connected algebra of infinite representation type. The following are equivalent.
\begin{itemize}
\item[(a)] $A$ is an $m$-cluster tilted algebra arising from a quiver of type $\tAa$.
\item[(b)] Is one of the following algebras
\begin{itemize}
\item [(i)] $A$ is an iterated tilted algebra of type $\tAa$, or
\item [(ii)] $A = \mathcal{R}_m(B)=B \semidirprod  \Ext^{m+1}_B(DB,B)$, is the $m$-relation extension of the algebra $B$, where $B = \End_{\Db(H)} (T)$ is an iterated tilted algebra of type $\tAa$ of global dimension $m+1$, with $T$ a tilting complex in the fundamental domain $\Ss$.
\end{itemize}
\end{itemize}
  \end{teorema}
\begin{proof}
$b)$ implies $a)$. It follows from \cite[Theorem 2.5]{FPT}.

\smallskip

$a)$ implies $b)$.
Let $A$ be a representation infinite $m$-cluster tilted algebra arising from a quiver of type $\tAa$. Then, by Remark \ref{quiverconrelaciones}, $A$ has a unique root cycle $\mathfrak{C}$ without zero relations an the other cycles are all saturated.

If $A$ has non saturated cycles, then $A$ has a root cycle and branches satisfiyng the same conditions given in \cite{AS} , by Remark \ref{quiverconrelaciones}. Then $A$ is an iterated tilted algebra of type  $\tAa$ by \cite{AS}.

Let $A$ be a representation infinite $m$-cluster tilted algebra arising from a quiver of type $\tAa$. Then, by Remark \ref{quiverconrelaciones}, $A$ has a unique root cycle $\mathfrak{C}$ without zero relations an the other cycles are all saturated. Then, if we kill the root cycle we get a product of $m$-cluster tilted algebras arising from $m$-cluster categories of type $\Aa$, see  \ref{quiverconrelaciones} and \cite{Mu}. More precisely, $B = A/<\mathfrak{C}> = B_1 \times \cdots \times B_n$ where each $B_i$ an $m$-cluster tilted algebra arising from from a quiver of type $\Aa$. Now, we cut exactly one arrow of each saturated cycle of $B_i$. Let $\alpha_{i_1}, \cdots, \alpha_{i_n}$ be the deleted arrows in each saturated cycle of $B_i$.

More precisely we make an admissible cut in the sense of \cite[Definition 4.6]{BFPPT}. In this way, we obtain $C_i$, an iterated tilted algebra of type $\Aa$ of global dimension at most $m+1$. In fact, by \ref{quiverconrelaciones} each $C_i$ is a gentle algebra which is a tree with at most $m-1$ consecutive relations, then by \cite{AH} we get that $C_i$ is an iterated tilted algebra of type $\Aa$.  Using the description given by \cite{Mu} and \cite[Remark 4.3]{FPT}, we have that $B_i = \mathcal{R}_m(C_i)$.

 Finally, we make the admissible cut, $A/ <\alpha_{i_j}> = B$. Using the description from \cite[Definition 7.5]{Gub} of the quiver with relations of the algebra $A$ and how the algebras $B_i$ are glued with the root cycle, in internal or external way in the sense of \cite{AS}. We get that $B$ is a representation infinite iterated tilted algebra of type $\tAa$
by \cite[Theorem A]{AS}, since the algebras $C_i$ are iterated tilted algebras of type $\Aa$ (branches in the sense of Assem and Skowronski). Observe that global dimension of $B$ is $m+1$. Considering $\mathcal{R}_m(B)$ we recover $A$ since in each branch $C_i$ we recover $B_i$, and the relations outside saturated branches are at most $m-1$. Then $ A = \mathcal{R}_m(B)$.\end{proof}

We illustrate the situation in the following example.

\begin{ejemplo}
Consider the representation infinite $3$-cluster tilted algebra $A$ arising from an $m$-cluster category of type ${\tilde{\mathbb{A}}}_{3,4}$ in Example \ref{repinfiniteAntilde}:

\begin{center}
\begin{tikzcd}[column sep=small]
5 \arrow[rr]&& 3  \\
1 \arrow[rr, shift left,"\beta"]
\arrow[rr, shift right,"\alpha_1"']
 && 2 \arrow[u,"\gamma"]\arrow[rr,"\alpha_2"'] && 4\arrow[d,"\alpha_3"]\\
&& 7\arrow[llu,"\alpha_5"] && 6\arrow[ll,"\alpha_4"] \end{tikzcd}with relations $\alpha_i \alpha_{i+1}=0$ for $i$ integer modulo $6$, and $\beta \gamma=0$.
\end{center}

Consider the admissible cut by the arrow $\alpha_6$, i.e $A/<\alpha_6>$,
 we obtain an iterated tilted algebra of type ${\tilde{\mathbb{A}}}_{3,4}$ of global dimension $4$.

\begin{center}
\begin{tikzcd}[column sep=small]
5 \arrow[rr]&& 3  \\
1 \arrow[rr, shift left,"\beta"]
\arrow[rr, shift right,"\alpha_1"']
 && 2 \arrow[u,"\gamma"]\arrow[rr,"\alpha_2"'] && 4\arrow[d,"\alpha_3"]\\
&& 7 && 6\arrow[ll,"\alpha_4"] \end{tikzcd}
with relations $\alpha_i \alpha_{i+1}=0$, and $\beta \gamma=0$.
\end{center}
\end{ejemplo}

\begin{remark}
The result in Theorem \ref{teorema-extensiones} does not hold for $m$-cluster tilted algebras of finite representation type. 
\end{remark}

\begin{ejemplo}\label{ejemplo 10} Let $A$ be the algebra given by quiver with relations $(Q,I)$ in Figure \ref{figura ejemplo-10}. Then $A$ is representation finite $m$-cluster tilted algebra arising form a quiver of type $\tAa$ which can not be realized as a trivial extension since it is a triangular
algebra, and it is not an iterated tilted algebra of any type since its Auslander--Reiten quiver has oriented cycles.
\end{ejemplo}

\begin{figure}[h!]
\centering
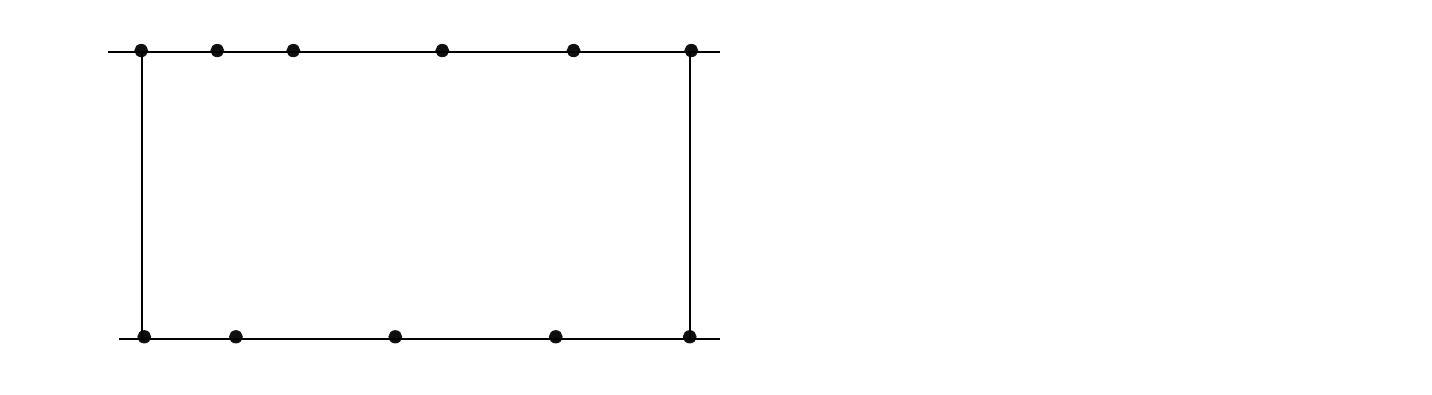
\def\svgwidth{2.8in}
\caption{Example \ref{ejemplo 10}}
\label{figura ejemplo-10}
\end{figure}

\section*{Acknowledgements}
The first and the third author were partially supported by ANPCyT and CONICET. The second author was also supported by FWF grant P30549.
The authors are grateful to the referee for his comments and suggestions which improve the presentation and some results of the paper.

{}


\end{document}